\documentclass[a4paper,11pt]{amsart} 
\usepackage{amsfonts, amsmath,amssymb,amsbsy,amstext,amsxtra,latexsym}
\usepackage[all]{xy}

\newcommand{\GV}{\Gamma_V}

\newtheorem{thm}{Theorem}[section]

\newtheorem{theorem}[thm]{Theorem}
\newtheorem*{theoremN}{Theorem}
\newtheorem{lemma}[thm]{Lemma}
\newtheorem{prop}[thm]{Proposition}

\newtheorem{rmk}[thm]{Remark}
\newtheorem*{rmkN}{Remark}

\title{Involutions on a surface of general type with $p_{g}=q=0$, $K^{2}=7$}
\author{Yongnam Lee and YongJoo Shin}
\date{}

\address{Department of Mathematics, Sogang University,
         Sinsu-dong, Mapo-gu, Seoul 121-742,  Korea}

\email{ynlee@sogang.ac.kr}

\address{Department of Mathematics, Sogang University,
         Sinsu-dong, Mapo-gu, Seoul 121-742, Korea}

\email{haushin@sogang.ac.kr}


\begin{document}
\maketitle
\begin{abstract}In this paper we study involutions on minimal surfaces of general type
with $p_g=q=0$ and $K^2=7$. We focus on the classification of the
birational models of the quotient surfaces and their branch divisors
induced by an involution.
\end{abstract}

\section{Introduction}
Algebraic surfaces of general type with vanishing geometric genus
have a very old history and have been studied by many
mathematicians. Since there are too many to mention here, we refer
a very recent survey \cite{BCP}. Nonetheless, a classification is
still lacking and it can be considered one of the most difficult
current problems in the theory of algebraic surfaces.

In the 1930s Campedelli \cite{SAP} constructed the first example of
a minimal surface of general type with $p_g=0$ using a double cover.
He used a double cover of $\mathbb{P}^2$ branched along a degree 10
curve with six points, not lying on a conic, all of which are a
triple point with another infinitely near triple point. After his
construction, the covering method has been one of main tools for
constructing new surfaces.

Surfaces of general type with $p_g=q=0, K^2=1$, and with an
involution have studied by Keum and the first named author \cite{FIG}, and
completed later by Calabri, Ciliberto and Mendes Lopes \cite{NGI}.
Also surfaces of general type with $p_g=q=0, K^2=2$, and with an
involution have studied by Calabri, Mendes Lopes, and Pardini
\cite{CMP}. Previous studies motivate the study of surfaces of
general type with $p_g=q=0, K^2=7$, and with an involution.

We know that a minimal surface of general type with $p_g=q=0$
satisfies $1\le K^{2}\le9$. One can ask whether there is a minimal
surface of general type with $p_{g}=q=0,$ and with an involution
whose quotient is birational to an Enriques surface. Indeed, there
are examples that are minimal surfaces of general type with
$p_{g}=q=0$, and $K^{2}=1,2,3,4$ constructed by a double cover of an
Enriques surface in \cite{Keum}, \cite{FIG}, \cite{ANFS},
\cite{SEC}. On the other hand, there is no a minimal surface of
general type with $p_{g}=q=0$ and $K^{2}=9$ (resp. $8$) having an
involution whose quotient is birational to an Enriques surface by
Theorem $4.3$ (resp. $4.4$) in \cite{RSMN}. Therefore, it is worth
to classify the possible branch divisors and to find an example
whose quotient is birational to an Enriques surface in the cases
$K^2=5,6,7$. We focus on the classification of branch divisors
induced by an involution instead of finding examples. We have only
two possible cases by excluding all other cases. Precisely, we prove
the following in Section 4.

\begin{theoremN} Let $S$ be a minimal surface of general type
with $p_{g}(S)=q(S)=0$, $K_{S}^{2}=7$ having an involution $\sigma$.
Suppose that the quotient $S/\sigma$ is birational to an Enriques
surface. Then the number of fixed points is 9, and the fixed divisor is a curve of genus 3 or
consists of two curves of genus 1 and 3. Furthermore, $S$
has a 2-torsion element.
\end{theoremN}

\medskip

Let $S$ be a minimal surface of general type with $p_{g}(S)=q(S)=0$
having an involution $\sigma$. There is a commutative diagram:
$\xymatrix{
  V  \ar[r]^{\epsilon}  \ar[d]_{\tilde{\pi}} & S  \ar[d]^{\pi}\\
  W \ar[r]^{\eta}                                 & \Sigma }$
In this diagram $\pi$ is the quotient map induced by the involution
$\sigma$. And $\epsilon$ is the blow-up of $S$ at $k$ isolated fixed
points of $\sigma$. Also, $\tilde{\pi}$ is induced by the quotient
map $\pi$ and $\eta$ is the minimal resolution of the $k$ double
points made by the quotient map $\pi$. And, there is a fixed divisor
$R$ of $\sigma$ on $S$ which is a smooth, possibly
reducible, curve. We set $R_{0}:=\epsilon^{*}(R)$ and
$B_{0}:=\tilde{\pi}(R_{0})$. Let $\Gamma_i$ be an irreducible
component of $B_0$. When we write $\substack{\Gamma_{i}\\(m, n)}$,
$m$ means $p_{a}(\Gamma_i)$ and $n$ is $\Gamma_i^2$.

\medskip

In the paper, we give the classification of the birational models of
the quotient surfaces and their branch divisors induced by an
involution when $K^2_S=7$. Precisely, we have the following table of
classification.
\begin{table}[h]
\begin{tabular}{|c|c|l|l|}
\hline
$k$ & $K_{W}^{2}$  &  $B_{0}$  &  $W$ \\
\hline
5 & 2  &  $\substack{\Gamma_{0}\\(1,-2)}$  &  minimal of general type \\
\hline 7  & 1  &  $\substack{\Gamma_{0}\\(3,2)}$  &  minimal
of general type  \\
\hline 7&  0  &
$\substack{\Gamma_{0}\\(2,-2)}$ &
minimal properly elliptic, or of general\\
&  &  $\substack{\Gamma_{0}\\(2,0)}+\substack{\Gamma_{1}\\(1,-2)}$ &  type whose minimal model has $K^2=1$\\
\hline 9& $-2$   &
$\substack{\Gamma_{0}\\(4,2)}+\substack{\Gamma_{1}\\(0,-4)}$ & $\kappa(W)\le 1$, and      \\
&      &  $\substack{\Gamma_{0}\\(3,-2)}$  & if $W$ is birational to an Enriques surface  \\
&      &
$\substack{\Gamma_{0}\\(4,4)}+\substack{\Gamma_{1}\\(1,-2)}+\substack{\Gamma_{2}\\(0,-4)}$
& then $B_0=\substack{\Gamma_{0}\\(3,0)}+\substack{\Gamma_{1}\\(1,-2)}$ or $\substack{\Gamma_{0}\\(3,-2)}$. \\
&      &  $\substack{\Gamma_{0}\\(4,4)}+\substack{\Gamma_{1}\\(0,-6)}$  &  \\
&      &  $\substack{\Gamma_{0}\\(3,0)}+\substack{\Gamma_{1}\\(1,-2)}$  &  \\
&      &  $\substack{\Gamma_{0}\\(3,2)}+\substack{\Gamma_{1}\\(2,0)}+\substack{\Gamma_{2}\\(0,-4)}$ &  \\
&      &  $\substack{\Gamma_{0}\\(3,2)}+\substack{\Gamma_{1}\\(1,-4)}$ &  \\
&      &  $\substack{\Gamma_{0}\\(2,-2)}+\substack{\Gamma_{1}\\(2,0)}$  &  \\
&      &  $\substack{\Gamma_{0}\\(3,2)}+\substack{\Gamma_{1}\\(1,-2)}+\substack{\Gamma_{2}\\(1,-2)}$  &  \\
&      & $\substack{\Gamma_{0}\\(2,0)}+\substack{\Gamma_{1}\\(2,0)}+\substack{\Gamma_{2}\\(1,-2)}$  &  \\
\hline
11 & $-4$ &  & rational surface \\
\hline
\end{tabular}
\end{table}

If $k=11$, the bicanonical map is composed with the involution. We
will omit the classification of $B_0$ for $k=11$ because there are
detailed studies in \cite{CSGNB}, \cite{NGI} and \cite{BS7}.

\medskip

The paper is organized as follows. In Section 3 we provide the
classification of branch divisors $B_{0}$, and birational models of
quotient surfaces $W$ for each possible $k$. Our approach follows by
the same approach as in \cite{NGI}, \cite{CMP} and \cite{RSMN}. But we have to
face different problems with respect to previous known results.
Section 4 is devoted to the study when $W$ is birational to an
Enriques surface. Firstly, we see that an Enriques surface $W'$,
obtained by contracting two $(-1)$-curves from $W$, has eight
disjoint $(-2)$-curves. Then via detailed study of Enriques surfaces
with eight $(-2)$-curves, only two possible cases of branch divisors
are remained by excluding all other cases. Section 5 is devoted to
the study of the branch divisors of examples given in \cite{BS7}.

Even if we are not able to construct a new example of such
surfaces which are double covers of surfaces birational to an Enriques surface or surfaces of
general type, our work will help to find such an example and to give
the classification of these surfaces.

\section{Notation and Conventions}
In this section we fix the notation which will be used. We work over the field of complex numbers in this paper.

Let $X$ be a smooth projective surface. Let $\Gamma$ be a curve in
$X$ and $\hat{\Gamma}$ be the normalization of $\Gamma$. We set:\\\\
$K_X$: the canonical divisor of $X$;\\
$NS(X)$: the N\'eron-Severi group of $X$;\\
$\rho(X)$: the rank of $NS(X)$;\\
$\kappa(X)$: the Kodaira dimension of $X$;\\
$q(X)$: the irregularity of $X$, that is, $h^{1}(X,\mathcal{O}_{X})$;\\
$p_{g}(X)$: the geometric genus of $X$, that is, $h^{0}(X,\mathcal{O}_{X}(K_{X}))$;\\
$p_{a}(\Gamma)$: the arithmetic genus of $\Gamma$, that is,
$\Gamma(\Gamma+K_{X})/2+1$;\\
$p_{g}(\Gamma)$: the geometric genus of $\Gamma$, that is,
$h^{0}(\hat{\Gamma},\mathcal{O}_{\hat{\Gamma}}(K_{\hat{\Gamma}}))$;\\
$\equiv$: the linear equivalence of divisors on a surface;\\
$\sim$: the numerical equivalence of divisors on a surface;\\
$\Gamma\colon(m,n)$ or $\substack{\Gamma\\(m, n)}$: $m$ is
$p_{a}(\Gamma)$ and $n$ is the self intersection number of $\Gamma$;\\
$(-n)$-curve: a smooth irreducible rational curve with the self intersection number $-n$,
in particular we call that a $(-2)$-curve is nodal;\\
We usually omit the sign $\cdot$ of the intersection product of two
divisors on a surface.\\

Let $S$ be a minimal surface of general type with $p_{g}(S)=q(S)=0$
having an involution $\sigma$. Then there is a commutative diagram:
\begin{displaymath}
\xymatrix{
  V  \ar[r]^{\epsilon}  \ar[d]_{\tilde{\pi}} & S  \ar[d]^{\pi}\\
  W \ar[r]^{\eta}                                 & \Sigma }
\end{displaymath}
In the above diagram $\pi$ is the quotient map induced by the
involution $\sigma$. And $\epsilon$ is the blowing-up of $S$ at $k$
isolated fixed points arising from the involution $\sigma$. Also,
$\tilde{\pi}$ is induced by the quotient map $\pi$ and $\eta$ is the
minimal resolution of the $k$ double points made by the quotient map
$\pi$.  We denote the $k$ disjoint $(-1)$-curves on $V$ (resp. the
$k$ disjoint $(-2)$-curves on $W$) related to the $k$ disjoint
isolated fixed points on $S$ (resp. the $k$ double points on
$\Sigma$) as $E_{i}$ (resp. $N_{i})$, $i=1,\dots, k$. And, there is
a fixed divisor $R$ of $\sigma$ on $S$ which is a
smooth, possibly reducible, curve. So we set
$R_{0}:=\epsilon^{*}(R)$ and $B_{0}:=\tilde{\pi}(R_{0})$.

The map $\tilde{\pi}$ is a flat double cover branched on
$\tilde{B}:=B_{0}+\sum_{i=1}^{k}N_{i}$. Thus there exists a divisor
$L$ on $W$ such that $2L\equiv\tilde{B}$ and
\[\tilde{\pi}_{*}\mathcal{O}_{V}=\mathcal{O}_{W}\oplus
\mathcal{O}_{W}(-L).\]\ Moreover,
$K_{V}\equiv\tilde{\pi}^{*}(K_{W}+L)$ and
$K_{S}\equiv\pi^{*}K_{\Sigma}+R$.

\medskip

\section{Classification of branch divisors and quotient surfaces}

In this section we focus on the classification of the
birational models of the quotient surfaces and their branch divisors
induced by an involution.

Since $\epsilon^{*}(2K_{S})\equiv\tilde{\pi}^{*}(2K_{W}+B_{0})$, the
divisor $2K_{W}+B_{0}$ is nef and big, and
$(2K_{W}+B_{0})^{2}=2K_{S}^{2}$. We begin with recalling the results in
\cite{NGI} and \cite{RSMN}.

\begin{prop}[Proposition 3.3 and Corollary 3.5 in \cite{NGI}]\label{prop:3.3}
Let $S$ be a minimal surface of general type with $p_{g}=0$ and let
$\sigma$ be an involution of $S$.
Then\\
(i) $k\geq4$;\\
(ii) $K_{W}L+L^{2}=-2$;\\
(iii) $h^{0}(W,\mathcal{O}_{W}(2K_{W}+L))=K_{W}^{2}+K_{W}L$;\\
(iv) $K_{W}^{2}+K_{W}L\geq0$;\\
(v) $k=K_{S}^{2}+4-2h^{0}(W,\mathcal{O}_{W}(2K_{W}+L))$;\\
(vi) $h^{0}(W,\mathcal{O}_{W}(2K_{W}+B_{0}))=K_{S}^{2}+1-h^{0}(W,\mathcal{O}_{W}(2K_{W}+L))$;\\
(vii) $K_{W}^{2}\geq K_{V}^{2}$.
\end{prop}

\begin{prop}[Corollary 3.6 in \cite{NGI}]\label{coro:3.6} Let $S$ be a minimal surface
of general type with $p_{g}=0$, let $\varphi\colon S\rightarrow
\mathbb{P}^{K_{S}^{2}}$ be the bicanonical map of $S$ and let
$\sigma$ be an involution of $S$.
Then the following conditions are equivalent:\\
(i) $\varphi$ is composed with $\sigma$;\\
(ii) $h^{0}(W,\mathcal{O}_{W}(2K_{W}+L))=0$;\\
(iii) $K_{W}(K_{W}+L)=0$;\\
(iv) the number of isolated fixed points of $\sigma$ is
$k=K_{S}^{2}+4$.
\end{prop}

By $(i)$ and $(v)$ of Proposition \ref{prop:3.3}, the possibilities
of $k$ are $5,7,9,11$ if $K_{S}^{2}=7$. In particular, if $k=11$,
the bicanonical map $\varphi$ is composed with the involution, which
is treated by Proposition \ref{coro:3.6}.

\begin{lemma}[Theorem 3.3 in \cite{RSMN}] \label{lemma:RSMN2}
Let $W$ be a smooth rational surface and let
$N_{1},\dots,N_{k}\subset W$ be disjoint nodal curves. Then\\
(i) $k\le \rho(W)-1$, and equality holds if and only if $W={\mathbb F}_2$;\\
(ii) if $k=\rho(W)-2$ and $\rho(W)\ge 5$, then $k$ is even.
\end{lemma}

\begin{lemma}[Proposition 4.1 in \cite{RSMN} and Remark 4.3 in \cite{PSMN}] \label{lemma:RSMN1}
Let $W$ be a surface with $p_{g}(W)=q(W)=0$ and $\kappa(W)\geq0$,
and let $N_{1},\dots,N_{k}\subset W$
be disjoint nodal curves. Then\\
(i) $k\le \rho(W)-2$ unless $W$ is a fake projective plane;\\
(ii) if $k=\rho(W)-2$, then $W$ is minimal unless $W$ is the blowing-up of
a fake projective plane at one point or at two infinitely near
points.
\end{lemma}

For simplicity of notation, we let $D$ stand for $2K_{W}+B_{0}$.

\begin{theorem}\label{theorem:k}
Let $S$ be a minimal surface of general type with $p_g(S)=0$ and
$K^2_S=7$ having an involution $\sigma$. Then\\
(i) $D^{2}=14$;\\
(ii) if $k=11$, then $K_{W}D=0$, $K^2_W=-4$, and $W$ is a rational surface;\\
(iii) if $k=9$, then $K_{W}D=2$, $K_{W}^{2}=-2$, and $\kappa(W)\le 1$;\\
(iv) if $k=7$, then $K_{W}D=4$, $0\le K_{W}^{2}\le 1$, and
$\kappa(W)\ge 1$. Furthermore, if $W$ is properly elliptic then
$K^2_W=0$. If $K^2_W=1$ then $W$ is minimal of general type. And
if $K^2_W=0$ and $W$ is of general type then $K_{W'}^2=1$ where $W'$
is the minimal model of $W$;\\
(v) if $k=5$, then $K_{W}D=6$, $K_{W}^{2}=2$, and $W$ is minimal of
general type.
\end{theorem}

\begin{proof}
$(i)$ This follows by $\epsilon^{*}(2K_{S})\equiv\tilde{\pi}^{*}(D)$ and $K_{S}^{2}=7$.\\
$(ii)$ Firstly, $K_{W}D=2K_{W}(K_{W}+L)=0$ by
Proposition \ref{coro:3.6}. Secondly,
$K_{V}^{2}=K_{S}^{2}-k=7-11=-4$. We have thus $K_{W}^{2}\ge -4$ by $(vii)$ of
Proposition \ref{prop:3.3}. Finally, $K_{W}^{2}\le 0$ by the
algebraic index theorem because $K_{W}D=0$ and $D$ is nef and big.
Since $K_WD=0$, $W$ can be a rational surface or birational to
an Enriques surface. Enriques surface is excluded by Theorem 3 in \cite{DBS}.
Also, by Lemma \ref{lemma:RSMN2} $k\le
\rho(W)-3$, and we have thus $\rho(W)\ge 14$. Therefore $K^2_W=-4$.   \\
$(iii)$ Firstly, $K_{W}D=2K_{W}(K_{W}+L)=2$ follows by
$(iii)$ and $(v)$ of Proposition \ref{prop:3.3}. Secondly,
$K_{V}^{2}=K_{S}^{2}-k=7-9=-2$. We have thus $K_{W}^{2}\geq -2$ by $(vii)$ of
Proposition \ref{prop:3.3}. Finally, the algebraic index theorem yields
$0\geq(7K_{W}-D)^{2}=49K_{W}^{2}-14K_{W}D+D^{2}=49K_{W}^{2}-14$, and we have thus $K_{W}^{2}\leq0$.

If $W$ is a rational surface then by Lemma \ref{lemma:RSMN2} $k\le
\rho(W)-3$, and so $\rho(W)\ge 12$. Therefore $K^2_W=-2$. If
$\kappa(W)\ge 0$ then by Lemma \ref{lemma:RSMN1} $\rho(W)\ge 11$. If
$\rho(W)=11$ then $W$ is minimal. It gives a contradiction
because $K^2_W=-1$. Therefore $\rho(W)=12$ and $K^2_W=-2$.

Moreover, $W$ is not of general type; suppose $W$ is of general type,
then we consider a birational morphism $t\colon$ $W\rightarrow W'$
such that $W'$ is the minimal model of $W$. Also, we can write
$K_{W}\equiv t^{*}(K_{W'})+E$, $E>0$ since $K_{W}^{2}\le0$.
Then $Dt^{*}(K_{W'})=2$; firstly, $Dt^{*}(K_{W'})\leq2$ because
$2=DK_{W}=Dt^{*}(K_{W'})+DE$ and $D$ is nef. Secondly,
$Dt^{*}(K_{W'})\geq2$ follows from that
$Dt^{*}(K_{W'})=2K_{W}t^{*}(K_{W'})+B_{0}t^{*}(K_{W'})=2(t^{*}(K_{W'})+E)t^{*}(K_{W'})
+B_{0}t^{*}(K_{W'})=2K_{W'}^{2}+B_{0}t^{*}(K_{W'})\geq2$ because
$K_{W'}^{2}>0$ and $K_{W'}$ is nef.

The algebraic index theorem yields
$0\geq(7t^{*}(K_{W})-D)^{2}=49t^{*}(K_{W'})^{2}-14Dt^{*}(K_{W'})+D^{2}=49K_{W'}^{2}-28+14.$
We have thus $K_{W'}^{2}\leq0$, which gives a contradiction. \\
$(iv)$ Since $K_V^2=K^2_S-k=0$, $K^2_W\ge 0$. $K_WD=4$ yields
$K_W^2\le 1$. $K^2_W\ge 0$ and $K_WD=4$ imply that $W$ is not
birational to an Enriques surface. Again $k=7$ implies that if $W$ is
a rational surface then $K^2_W=0$. But then $h^0(W,\mathcal{O}_{W}(-K_W))>0$ and this
is impossible because $D$ is nef.

If $W$ is properly elliptic then $K^2_W=0$. And if $K_{W}^2=1$ then $W$
is a minimal surface of general type by Lemma \ref{lemma:RSMN1}.

Now suppose that $K^2_W=0$ and $W$ is of general type. Then we
consider a birational morphism $t\colon$ $W\rightarrow W'$ such that
$W'$ is the minimal model of $W$. Suppose $K_{W'}^2\ge 2$.

We write $K_{W}\equiv t^{*}(K_{W'})+E$, $E>0$. Firstly,
$Dt^{*}(K_{W'})\le 4$ because $K_WD=4$. Secondly, $Dt^{*}(K_{W'})\ge
4$: $Dt^{*}(K_{W'})=2K_{W}t^{*}(K_{W'})+B_{0}t^{*}(K_{W'})
=2(t^{*}(K_{W'})+E)t^{*}(K_{W'}) +B_{0}t^{*}(K_{W'})
=2K_{W'}^{2}+B_{0}t^{*}(K_{W'})\ge 4$ since we suppose
$K_{W'}^{2}\ge 2$ and $K_{W'}$ is nef.

Therefore $Dt^{*}(K_{W'})=4$. Then by the algebraic index theorem
and $D^2=14$,
$0\geq(7t^{*}(K_{W'})-2D)^{2}=49t^{*}(K_{W'})^{2}-28Dt^{*}(K_{W'})+4D^{2}=49K_{W'}^{2}-112+56,$
which gives a contradiction. \\
$(v)$ Since $K^2_V=2$, $K_W^2\ge 2$ and so $W$ is either a rational
surface or a surface of general type. But if it is a rational
surface then $h^0(W,\mathcal{O}_{W}(-K_W))>0$ gives a contradiction. Also, $K_WD=6$ and
the algebraic index theorem implies that $K^2_W\le 2$.

Now we know that $W$ is of general type with $K^2_W=2$, it is enough
to prove that $W$ is minimal. Suppose $W$ is not minimal. Then we
consider a birational morphism $t\colon$ $W\rightarrow W'$ such that
$W'$ is the minimal model of $W$. Also, we can write $K_{W}\equiv
t^{*}(K_{W'})+E$, $E>0$. Firstly, $Dt^{*}(K_{W'})\le 6$ because
$K_WD=6$, and $K_{W'}^2\ge 3$. Secondly, $Dt^{*}(K_{W'})\ge 6$:
$Dt^{*}(K_{W'})=2K_{W}t^{*}(K_{W'})+B_{0}t^{*}(K_{W'})
=2(t^{*}(K_{W'})+E)t^{*}(K_{W'}) +B_{0}t^{*}(K_{W'})
=2K_{W'}^{2}+B_{0}t^{*}(K_{W'})\ge 6$ since $K_{W'}^{2}\ge 3$ and
$K_{W'}$ is nef.

Therefore $Dt^{*}(K_{W'})=6$. Then by the algebraic index theorem
and $D^2=14$,
$0\geq(7t^{*}(K_{W'})-3D)^{2}=49t^{*}(K_{W'})^{2}-42Dt^{*}(K_{W'})+9D^{2}=49K_{W'}^{2}-252+126,$
which gives a contradiction.
\end{proof}


We now study the possibilities of an irreducible component
$\Gamma\subset B_{0}$ for each number of isolated fixed points. Let
$\GV$ be the preimage of $\Gamma$ in the double cover $V$ of $W$. We
do not consider the case $k=11$ because it is already treated in \cite{CSGNB}, \cite{NGI} and
\cite{BS7}.

\begin{lemma}\label{lemma:2kv.d} For any irreducible component $\Gamma\subset B_{0}$ on $W$,
$2K_{V}\GV=\Gamma D$, where $\tilde{\pi}^{*}\Gamma \equiv 2\GV$.
\end{lemma}

\begin{proof} We have $2\Gamma D=\tilde{\pi}^{*}(\Gamma)
\tilde{\pi}^{*}(D)=2\GV\epsilon^{*}(2K_{S})$. We have thus $\Gamma
D=\GV\epsilon^{*}(2K_{S})$. On the other hand, we know that
$\GV\epsilon^{*}(2K_{S})=2K_{V}\GV$ because $\GV$ $\cap$
(Exceptional locus of $\epsilon)= \varnothing$. Therefore $2K_{V}\GV=\Gamma
D$.
\end{proof}

\begin{rmk}
\emph{By Lemma \ref{lemma:2kv.d}, $\Gamma D$ should be even, and if $\Gamma D=0$ then $\Gamma$ is a $(-4)$-curve.}
\end{rmk}

\subsection{\bf Classification of $B_{0}$ for $k=9$}\label{CB:k=9}

In this case, $B_{0}D=10$ because $B_{0}D=(D-2K_{W})D=14-4=10$. So
$\Gamma D=10, 8, 6, 4$, or $2$.

\paragraph{1) The case $\Gamma D=10$.}
Since $D^{2}=14$ and $D$ is nef and big,
$0\geq(7\Gamma-5D)^{2}=49\Gamma^{2}-70\Gamma
D+25D^{2}=49\Gamma^{2}-350$ by the algebraic index theorem. That is,
$\Gamma^{2}\leq7$. Thus we get $\GV^{2}\leq3$ because
$2\GV^{2}=\Gamma^{2}$. Moreover,  we know that $0\leq
p_{a}(\GV)=1+\frac{1}{2}(\GV^{2}+K_{V}\GV)
=1+\frac{1}{2}(\GV^{2}+5)$ by Lemma \ref{lemma:2kv.d}. Thus
$-7\leq\GV^{2}\leq3$. By the genus formula,
$\GV^{2}=-7,-5,-3,-1,1,3$.

\begin{enumerate}
\item The case $\GV^{2}=-7$: in this case, $p_{a}(\GV)=0$. So $\Gamma\colon(0,-14)$.
Then if we write that
$B_{0}=\Gamma_{0}+\Gamma_{1}+\cdots+\Gamma_{l}$ such that
$\Gamma_{0}=\Gamma$ and $\Gamma_{i}$ are $(-4)$-curves for each
$i=1,\dots ,l$, then
\[6=2-2K_{W}^{2}=K_{W}(D-2K_{W})=K_{W}B_{0}=12+2l.\]
We get a contradiction because $l=-3$
\item The cases $\GV^{2}=-5, -3$: similar arguments as the case (1) give contradictions because $l<0$.
\item The case $\GV^{2}=-1$: we get $p_{a}(\GV)=3$. So $\Gamma\colon(3,-2)$ and $l=0$.
\item The case $\GV^{2}=1$: here, $p_{a}(\GV)=4$. So $\Gamma\colon(4,2)$ and $l=1$.
\item The case $\GV^{2}=3$: lastly, $p_{a}(\GV)=5$. So $\Gamma\colon(5,6)$ and $l=2$.
\end{enumerate}
We have thus the following possibilities of $B_{0}$ in the case
$\Gamma D=10$.
\[B_0:\,
\substack{\Gamma_{0}\\(5,6)}+\substack{\Gamma_{1}\\(0,-4)}+\substack{\Gamma_{2}\\(0,-4)},
\, \substack{\Gamma_{0}\\(4,2)}+\substack{\Gamma_{1}\\(0,-4)}, \,
\substack{\Gamma_{0}\\(3,-2)}
\]

\begin{rmk}
\emph{$\substack{\Gamma_{0}\\(5,6)}+\substack{\Gamma_{1}\\(0,-4)}+\substack{\Gamma_{2}\\(0,-4)}$
cannot occur by Proposition $2.1.1$ of \cite{MQS} because a smooth
rational curve in $B_{0}$ corresponds to a smooth rational curve on
$S$.}
\end{rmk}

\paragraph{2) The case $\Gamma_{0}D=8$ and $\Gamma_{1}D=2$}
We can use the similar argument as the above \ref{CB:k=9}.$1)$ for each of $\Gamma_{0}D$ and
$\Gamma_{1}D$. However, we have to consider
$B_{0}=\Gamma_{0}+\Gamma_{1}+\Gamma_{1}'+\cdots+\Gamma_{l}'$ to get
the possibilities for $B_{0}$, where $\Gamma_{i}'\colon(0,-4)$ for
all $i\in \{1,2,\dots ,l\}$ if those exist. Then we get the following
possible cases.
\[B_{0}:\,
\substack{\Gamma_{0}\\(4,4)}+\substack{\Gamma_{1}\\(1,-2)}+\substack{\Gamma_{2}\\(0,-4)},
\, \substack{\Gamma_{0}\\(4,4)}+\substack{\Gamma_{1}\\(0,-6)}, \,
\substack{\Gamma_{0}\\(3,0)}+\substack{\Gamma_{1}\\(1,-2)} \]

Now, we give all remaining cases by the similar argument as the above \ref{CB:k=9}.$2)$.

\paragraph{3) The case $\Gamma_{0}D=6$ and $\Gamma_{1} D=4$}\textrm{ }
\[B_{0}: \,
\substack{\Gamma_{0}\\(3,2)}+\substack{\Gamma_{1}\\(2,0)}+\substack{\Gamma_{2}\\(0,-4)},
\, \substack{\Gamma_{0}\\(3,2)}+\substack{\Gamma_{1}\\(1,-4)}, \,
\substack{\Gamma_{0}\\(2,-2)}+\substack{\Gamma_{1}\\(2,0)}\]
\paragraph{4) The case $\Gamma_{0}D=6$, $\Gamma_{1} D=2$ and $\Gamma_{2} D=2$}\textrm{ }
\[B_{0}: \,
\substack{\Gamma_{0}\\(3,2)}+\substack{\Gamma_{1}\\(1,-2)}+\substack{\Gamma_{2}\\(1,-2)}\]
\paragraph{5) The case $\Gamma_{0}D=4$, $\Gamma_{1} D=4$ and $\Gamma_{2} D=2$}
\textrm{ }
\[B_{0}: \,
\substack{\Gamma_{0}\\(2,0)}+\substack{\Gamma_{1}\\(2,0)}+\substack{\Gamma_{2}\\(1,-2)}\]
\paragraph{6) The case $\Gamma_{0}D=4$, $\Gamma_{1} D=2$, $\Gamma_{2} D=2$ and $\Gamma_{3} D=2$}
\textrm{ }\\
We get a contradiction by the similar argument in
\ref{CB:k=9}.$1).(1)$.

\paragraph{7) The case $\Gamma_{0}D=2$, $\Gamma_{1} D=2$, $\Gamma_{2} D=2$, $\Gamma_{3} D=2$ and $\Gamma_{4} D=2$}
\textrm{ }\\
This case is also ruled out by the similar argument in
\ref{CB:k=9}.$1).(1)$.
\medskip

By Theorem \ref{theorem:k} and from the above classification, we get
Table 1:
\begin{table}[h]
\centering \textrm{Table 1: Classifications of $K_{W}^{2}$, $B_{0}$
and $W$ for $k=9$\\}\textrm{\\}
\begin{tabular}{|c|l|l|}
\hline
$K_{W}^{2}$  &  $B_{0}$  &  $W$\\
\hline
$-2$   &  $\substack{\Gamma_{0}\\(4,2)}+\substack{\Gamma_{1}\\(0,-4)}$ & $\kappa(W)\le 1$  \\
      &  $\substack{\Gamma_{0}\\(3,-2)}$  &  \\
      &  $\substack{\Gamma_{0}\\(4,4)}+\substack{\Gamma_{1}\\(1,-2)}+\substack{\Gamma_{2}\\(0,-4)}$  &  \\
      &  $\substack{\Gamma_{0}\\(4,4)}+\substack{\Gamma_{1}\\(0,-6)}$  &  \\
      &  $\substack{\Gamma_{0}\\(3,0)}+\substack{\Gamma_{1}\\(1,-2)}$  &  \\
      &  $\substack{\Gamma_{0}\\(3,2)}+\substack{\Gamma_{1}\\(2,0)}+\substack{\Gamma_{2}\\(0,-4)}$ &  \\
      &  $\substack{\Gamma_{0}\\(3,2)}+\substack{\Gamma_{1}\\(1,-4)}$  &  \\
      &  $\substack{\Gamma_{0}\\(2,-2)}+\substack{\Gamma_{1}\\(2,0)}$  &  \\
      &  $\substack{\Gamma_{0}\\(3,2)}+\substack{\Gamma_{1}\\(1,-2)}+\substack{\Gamma_{2}\\(1,-2)}$  &  \\
      & $\substack{\Gamma_{0}\\(2,0)}+\substack{\Gamma_{1}\\(2,0)}+\substack{\Gamma_{2}\\(1,-2)}$  &  \\
\hline
\end{tabular}
\end{table}

\subsection{\bf Classification of $B_{0}$ for $k=7$}\label{classification:k=7}

In this case, $B_{0}D=6$. So $\Gamma D$ can be $6, 4, 2$. By using
similar arguments as the above \ref{CB:k=9}, we get the following
tables related to $K_{W}^{2}$ and $B_{0}$ for each case of $\Gamma
D$.

\paragraph{1) The case $\Gamma D=6$}
\textrm{ }

\medskip

\begin{tabular}{|c|l|}
\hline
$K_{W}^{2}$  &  $B_{0}$\\
\hline
1  &  $\substack{\Gamma_{0}\\(3,2)}$\\
\hline
0  &  $\substack{\Gamma_{0}\\(3,2)}+\substack{\Gamma_{1}\\(0,-4)}$, $\substack{\Gamma_{0}\\(2,-2)}$\\
\hline
\end{tabular}

\begin{lemma}\label{lemma:k=7}
$B_0=\substack{\Gamma_{0}\\(3,2)}+\substack{\Gamma_{1}\\(0,-4)}$ is
not possible.
\end{lemma}

\begin{proof}
by Theorem \ref{theorem:k}, $W$ is minimal properly elliptic, or of general
type whose minimal model $W'$ has $K^2_{W'}=1$ . If
$W$ is minimal properly elliptic, then we get a contradiction by Miyaoka's
theorem in \cite{MQS} because $W$ has seven disjoint $(-2)$-curves and
one $(-4)$-curve.

We now suppose that $W$ is of general type whose minimal model $W'$ has
$K^2_{W'}=1$. We consider a birational morphism $t: W\rightarrow W'$, and $K_{W}\equiv
t^{*}(K_{W'})+E$, where $E$ is the unique $(-1)$-curve. $E$
cannot meet seven disjoint $N_i$ because $K_{W'}t(N_{i})=-N_{i}E$
and $K_{W'}$ is nef. And $\Gamma_1E\le 1$ because $K_WB_0=4$,
$K_W\Gamma_0=2$, and $t^{*}(K_{W'})\Gamma_1\ge 1$. Then, Miyaoka's
theorem \cite{MQS} again gives a contradiction because $W'$ has
seven disjoint $(-2)$-curves, and one $(-4)$-curve or one
$(-3)$-curve.
\end{proof}

\medskip

\paragraph{2) The case $\Gamma_{0}D=4$ and $\Gamma_{1}D=2$}
\textrm{ }

\medskip

\begin{tabular}{|c|l|}
\hline
$K_{W}^{2}$  &  $B_{0}$\\
\hline
0  &  $\substack{\Gamma_{0}\\(2,0)}+\substack{\Gamma_{1}\\(1,-2)}$\\
\hline
\end{tabular}

\medskip

\paragraph{3) The case $\Gamma_{0}D=2$, $\Gamma_{1} D=2$ and $\Gamma_{2} D=2$}
\textrm{ }

This case is not possible by the similar argument in
\ref{CB:k=9}.$1).(1)$.

\medskip

\begin{table}[h]
\centering \textrm{Table 2: Classifications of $K_{W}^{2}$, $B_{0}$
and $W$ for $k=7$\\} \textrm{\\}
\begin{tabular}{|c|l|l|}
\hline
$K_{W}^{2}$  &  $B_{0}$  &  $W$\\
\hline
1  &  $\substack{\Gamma_{0}\\(3,2)}$  &  minimal of general type\\
\hline
0  & $\substack{\Gamma_{0}\\(2,-2)}$ &  minimal properly elliptic, or of general \\
  &  $\substack{\Gamma_{0}\\(2,0)}+\substack{\Gamma_{1}\\(1,-2)}$ & type whose  minimal model has $K^2=1$\\
 \hline
\end{tabular}
\end{table}


\subsection{\bf Classification of $B_{0}$ for $k=5$}\label{classification:k=5}

In this case, $B_{0}D=2$. So $\Gamma D$ can be $2$. By using similar
arguments as the above \ref{CB:k=9}, we get the following table
related to $K_{W}^{2}$ and $B_{0}$ for $\Gamma D$.

\begin{table}[h]
\centering \textrm{Table 3: Classifications of $K_{W}^{2}$, $B_{0}$
and $W$ for $k=5$\\} \textrm{\\}
\begin{tabular}{|c|l|l|}
\hline
$K_{W}^{2}$  &  $B_{0}$  &  $W$\\
\hline
2  &  $\substack{\Gamma_{0}\\(1,-2)}$  &  of general type\\
\hline
\end{tabular}
\end{table}


\medskip

\section{Quotient surface birational to an Enriques surface}\label{ctm}
In this section we treat the case when $W$ is birational to an
Enriques surface.

\begin{theorem}\label{prop:Enriques}
Let $S$ be a minimal surface of general type with $p_g(S)=0$ and
$K^2_S=7$ having an involution $\sigma$. If $W$ is birational to an Enriques surface then $k=9,
K^2_W=-2$, and the branch divisor
$B_0=\substack{\Gamma_{0}\\(3,0)}+\substack{\Gamma_{1}\\(1,-2)}$ or $\substack{\Gamma_{0}\\(3,-2)}$.
Furthermore, $S$ has a $2$-torsion element.
\end{theorem}

Suppose $W$ is birational an Enriques surface. Then by Theorem
\ref{theorem:k}, we have $k=9$ and $K_{W}^{2}=-2$. Consider the
contraction maps:
\[W\stackrel{\varphi_{1}}{\longrightarrow}W_{1}\stackrel{\varphi_{2}}{\longrightarrow}W',\]
where $\bar{E}_{1}$ is a $(-1)$-curve on $W$, $\bar{E}_{2}$ is
a $(-1)$-curve on $W_{1}$, $\varphi_{i}$ is the contraction of the
$(-1)$-curve $\bar{E}_{i}$, and $W'$ is an Enriques surface.

\begin{lemma}\label{lemma1:k=9}
i) $N_{i}\cap\bar{E}_{1}\neq\varnothing$ for some $i\in\{1,2,\dots,9\}$.\\
ii) $N_{1}\bar{E}_{1}=1$ after relabeling $\{N_{1},\dots,N_{9}\}$.\\
iii) $N_{s}\bar{E}_{1}=0$ for all $s\in\{2,\dots,9\}$.
\end{lemma}
\begin{proof} $i)$ Suppose that $N_{i}\cap\bar{E}_{1}=\varnothing$ for all $i=1,\dots,9$.
Let $A$ be the number of disjoint $(-2)$-curves on $W_{1}$. Then by Lemma \ref{lemma:RSMN1}
$(i)$, $9\le A\le \rho(W_{1})-2=9$. Thus $A=9$ and $W_{1}$ should be a minimal surface
by Lemma \ref{lemma:RSMN1} $(ii)$. This is a contradiction because $W_{1}$ is not minimal.
Hence $N_{i}\cap\bar{E}_{1}\neq\varnothing$ for some $i\in\{1,2,\dots,9\}$.\\
$ii)$ By part $i)$ we may choose a $(-2)$-curve $N_{1}$ such that $N_{1}\bar{E}_{1}=\alpha>0.$
Then $(\varphi_{1}(N_{1}))^{2}=-2+\alpha^{2}$ and $\varphi_{1}(N_{1})K_{W_{1}}=-\alpha$.
We claim that $\alpha$ must be 1. Indeed, suppose $\alpha\geq2$, then $(\varphi_{1}(N_{1}))^{2}>0$,
so $\varphi_{2}\circ\varphi_{1}(N_{1})$ is a curve on $W'$.
Moreover, $\varphi_{2}\circ\varphi_{1}(N_{1})K_{W'}\leq\varphi_{1}(N_{1})K_{W_{1}}$.
But the left side is zero because $2K_{W'}\equiv 0$ and the right side is negative
because $\varphi_{1}(N_{1})K_{W_{1}}=-\alpha$ by our assumption.
This is a contradiction, thus $\alpha=1$.\\
$iii)$ Suppose that $N_{s}\bar{E}_{1}\neq0$ for some
$s\in\{2,\dots,9\}$. Then $W_{1}$ would contain a pair of $(-1)$-curves with nonempty intersection. This is
impossible because $K_{W'}$ is nef. Hence $N_{s}\bar{E}_{1}=0$ for
all $s\in\{2,\dots,9\}$. \end{proof}

In this situation, consider an irreducible nonsingular curve $\Gamma$ disjoint to $N_{1}$
and such that $\bar{E}_{1}\Gamma=\beta$. Then we obtain the following.

\begin{lemma}\label{criterion of alpha}
$2p_{a}(\Gamma)-2=\Gamma^{2}+2\beta$.
\end{lemma}
\begin{proof} By Lemma \ref{lemma1:k=9},
\[K_{W}\equiv\varphi_{1}^{*}(K_{W_{1}})+\bar{E}_{1}\equiv
\varphi_{1}^{*}(\varphi_{2}^{*}(K_{W'})+\bar{E}_{2})
+\bar{E}_{1}\equiv\varphi_{1}^{*}\circ\varphi_{2}^{*}(K_{W'})+N_{1}+2\bar{E}_{1}.\]
So
$K_{W}\Gamma=\varphi_{1}^{*}\circ\varphi_{2}^{*}(K_{W'})\Gamma+N_{1}\Gamma
+2\bar{E}_{1}\Gamma=2\beta$ since $2K_{W'}\equiv0$ and $N_{1}$ and $\Gamma$ are disjoint.
Thus we get $2p_{a}(\Gamma)-2=\Gamma^{2}+2\beta$. \end{proof}

By referring to Table $1.$ of Section \ref{CB:k=9} with respect to
$K^{2}_{W}=-2$ and $k=9$, we obtain a list of possible branch curves
$B_{0}$. Then we can consider $\Gamma$ as one of the components
$\Gamma_{i}$ in the $B_{0}$. The possibilities for $\Gamma$ which we will consider are:
\[(0,-4), (2,-2), (2,0), (1,-2), (0,-6), (3,2),
(1,-4).\] We treat each case separately.
\begin{itemize}
\item[a)] The case $\Gamma\colon(0,-4)$\\
By Lemma \ref{criterion of alpha}, $\beta=1$. Thus $W'$ should
contain nine disjoint $(-2)$-curves. This is a contradiction
because $W'$ can contain at most eight disjoint $(-2)$-curves since it is an Enriques surface.
\end{itemize}

From now on, we consider the nodal Enriques surface $\Sigma'$
obtained by contracting eight $(-2)$-curves $\tilde{N_{i}}$,
$i=2,\dots,9$, where
$\tilde{N_{i}}:=\varphi_{2}\circ\varphi_{1}(N_{i})$ on $W'$. The
surface $\Sigma'$ has eight nodes $q_{i}$, $i=2,\dots,9$ and
$\tilde{\Gamma}_{\Sigma'}$ which is image of $\tilde{\Gamma}$, where
$\tilde{\Gamma}:=\varphi_{2}\circ\varphi_{1}(\Gamma)$ on $W'$. By
Theorem $4.1$ in \cite{ES8}, $\Sigma'=D_{1}\times D_{2}/G$, where
$D_{1}$ and $D_{2}$ are elliptic curves and $G$ is a finite group
$\mathbb{Z}_{2}^{2}$ or $\mathbb{Z}_{2}^{3}$. Let $p$ be the
quotient map $D_{1}\times D_{2}\longrightarrow D_{1}\times
D_{2}/G=\Sigma'$. The map $p$ is \'etale outside the preimage of
nodes $q_{i}$ on $\Sigma'$, and we note that
$\tilde{\Gamma}_{\Sigma'}$ does not meet with any eight nodes
$q_{i}$ on $\Sigma'$. We write $\hat{\Gamma}_{D_{1}\times D_{2}}$
for a component of $p^{-1}(\tilde{\Gamma}_{\Sigma'})$.
\begin{itemize}
\item[b)] The case $\Gamma\colon(0,-6)$\\
By Lemma \ref{criterion of alpha}, $\beta=2$. So $\tilde{\Gamma}$ is (2,2).
Then the normalization $\hat{\Gamma}^{nor}$ of $\hat{\Gamma}_{D_{1}\times D_{2}}$
is a smooth rational curve since $p_{a}(\Gamma)=0$ and $\Gamma$ is smooth.

Let $pr_{i}$ be the projection map $D_{1}\times D_{2}\longrightarrow
D_{i}$. Then this induces morphisms
$p_{i}\colon\hat{\Gamma}^{nor}\longrightarrow D_{i}$ which factors
through $pr_{i}|_{\hat{\Gamma}_{D_{1}\times D_{2}}}$. Then since $\hat{\Gamma}_{D_{1}\times D_{2}}$ is a
curve on $D_{1}\times D_{2}$, $p_{i}$ should be a surjective
morphism for some $i\in\{1,2\}$. However, this is impossible because
$p_{g}(\hat{\Gamma}^{nor})=0$ and $p_{g}(D_{i})=1$.

\item[c)] The case $\Gamma\colon(1,-4)$\\
By Lemma \ref{criterion of alpha}, $\beta=2$. So
$\tilde{\Gamma}$ is (3,4). Then the normalization
$\hat{\Gamma}^{nor}$ of $\hat{\Gamma}_{D_{1}\times D_{2}}$ is a
smooth elliptic curve because $p_{a}(\Gamma)=1$ and $\Gamma$ is
smooth. Thus $\hat{\Gamma}^{nor}\longrightarrow D_{1}\times D_{2}$
is a morphism of Abelian varieties and so must be linear, which
implies that $\hat{\Gamma}_{D_{1}\times D_{2}}$ is smooth. Thus
$\tilde{\Gamma}_{\Sigma'}$ is also smooth because
$\tilde{\Gamma}_{\Sigma'}$ does not meet any of the eight nodes
$q_{i}$ on $\Sigma'$ and $p$ is \'etale on away from the nodes
$q_{i}$. This is a contradiction since we assumed
$\tilde{\Gamma}_{\Sigma'}$ to be singular.

\item[d)] The case
$\substack{\Gamma_{0}\\(3,2)}+\substack{\Gamma_{1}\\(1,-2)}+\substack{\Gamma_{2}\\(1,-2)}$\\
By Lemma \ref{criterion of alpha}, we have
$\bar{E}_{1}\Gamma_{i}=1$ for $i=0,1,2$. So we get
$\tilde{\Gamma}_{0}\colon(3,4)$, $\tilde{\Gamma}_{1}\colon(1,0)$,
$\tilde{\Gamma}_{2}\colon(1,0)$ and
$\tilde{\Gamma}_{i}\tilde{\Gamma}_{j}=2$ for $i\neq j$ on the
Enriques surface $W'$. Now, we apply Proposition $3.1.2$ of
\cite{ESI} to the curve $\tilde{\Gamma}_{2}$. Then one of the linear
systems $|\tilde{\Gamma}_{2}|$ or $|2\tilde{\Gamma}_{2}|$ gives an
elliptic fibration $f\colon W'\longrightarrow\mathbb{P}^{1}$. So we
have the reducible non-multiple degenerate fibres
$\tilde{T}_{1}(=\tilde{N}_{2}+\tilde{N}_{3}+\tilde{N}_{4}+\tilde{N}_{5}+2E_{1})$
and
$\tilde{T}_{2}(=\tilde{N}_{6}+\tilde{N}_{7}+\tilde{N}_{8}+\tilde{N}_{9}+2E_{2})$
of $f$ by Theorem $5.6.2$ of \cite{ESI}, since $W'$ has eight
disjoint $(-2)$-curves. Moreover, $f$ has two double fibres $2F_{1}$
and $2F_{2}$ since $W'$ is an Enriques surface.

\begin{enumerate}
\item
Suppose $|\tilde{\Gamma}_{2}|$ determines the elliptic fibration.
Then $\tilde{\Gamma}_{2}$ is a fibre of $f$. Since
$\tilde{\Gamma}_{1}\tilde{\Gamma}_{2}=2$ (they meet at a point with
multiplicity 2), $2F_{1}\tilde{\Gamma}_{1}=2$,
$2F_{2}\tilde{\Gamma}_{1}=2$ and $\tilde{T}_{i}\tilde{\Gamma}_{1}=2$
for $i=1, 2$, we apply Hurwitz's formula to the covering
$f|_{\tilde{\Gamma}_{1}}\colon
\tilde{\Gamma}_{1}\longrightarrow\mathbb{P}^{1}$ to obtain
\[0=2p_{g}(\tilde{\Gamma}_{1})-2\ge2(-2)+5=1\] which is impossible.

\item
Suppose $|2\tilde{\Gamma}_{2}|$ determines the elliptic fibration.
Then $2\tilde{\Gamma}_{2}$ is a fibre of $f$. Since
$2F_{1}\tilde{\Gamma}_{1}(=2\tilde{\Gamma}_{2}\tilde{\Gamma}_{1})=4$,
$2F_{2}\tilde{\Gamma}_{1}=4$ and $\tilde{T}_{i}\tilde{\Gamma}_{1}=4$
for $i=1, 2$, we apply Hurwitz's formula to the covering
$f|_{\tilde{\Gamma}_{1}}\colon\tilde{\Gamma}_{1}\longrightarrow\mathbb{P}^{1}$
to obtain \[0=2p_{g}(\tilde{\Gamma}_{1})-2\ge4(-2)+3+2+2+2=1,\]
which is impossible.
\end{enumerate}

\item[e)] The case
$\substack{\Gamma_{0}\\(2,0)}+\substack{\Gamma_{1}\\(2,0)}+\substack{\Gamma_{2}\\(1,-2)}$\\
By Lemma \ref{criterion of alpha}, $\bar{E}_{1}\Gamma_{i}=1$ for $i=0,1,2$. So we have
$\tilde{\Gamma}_{0}\colon(2,2)$, $\tilde{\Gamma}_{1}\colon(2,2)$,
$\tilde{\Gamma}_{2}\colon(1,0)$ and $\tilde{\Gamma}_{i}\tilde{\Gamma}_{j}=2$ for $i\neq j$
on the Enriques surface $W'$.
\begin{lemma}\label{hgamma0=2}
$h^{0}(W',\mathcal{O}_{W'}(\tilde{\Gamma}_{1}))=2.$
\end{lemma}
\begin{proof}Since $2K_{W'}\equiv0$ and $K_{W'}+\tilde{\Gamma}_{1}$ is nef
and big,
\begin{eqnarray}
h^{i}(W',\mathcal{O}_{W'}(\tilde{\Gamma}_{1}))
&=&h^{i}(W',\mathcal{O}_{W'}(2K_{W'}+\tilde{\Gamma}_{1}))\nonumber\\
&=&h^{i}(W',\mathcal{O}_{W'}(K_{W'}+(K_{W'}+\tilde{\Gamma}_{1})))\nonumber\\
&=&0\nonumber
\end{eqnarray}
for $i=1,2$ by Kawamata-Viehweg Vanishing Theorem. Thus
\[h^{0}(W',\mathcal{O}_{W'}(\tilde{\Gamma}_{1}))=2\]
by Riemann-Roch Theorem.\end{proof}

\begin{lemma}\label{connectedness}
Let $T$ be a nef and big divisor on $W'$. \\
Then any divisor $U$ in a linear system $|T|$ is connected.
\end{lemma}
\begin{proof}
Consider an exact sequence
\begin{displaymath}
0\longrightarrow \mathcal{O}_{W'}(-U)\longrightarrow \mathcal{O}_{W'}
\longrightarrow \mathcal{O}_{U}\longrightarrow 0.
\end{displaymath}
Then we get $H^{0}(\mathcal{O}_{W'})\cong H^{0}(\mathcal{O}_{U})$
by the long exact sequence for cohomology, and so $U$ is connected.
\end{proof}

Now, we apply Proposition $3.1.2$ of \cite{ESI} to the curve
$\tilde{\Gamma}_{2}$. Then one of the linear systems
$|\tilde{\Gamma}_{2}|$ or $|2\tilde{\Gamma}_{2}|$ gives an elliptic
fibration $f:W'\longrightarrow\mathbb{P}^{1}$. So we have the
reducible non-multiple degenerate fibres
$\tilde{T}_{1}(=\tilde{N}_{2}+\tilde{N}_{3}+\tilde{N}_{4}+\tilde{N}_{5}+2E_{1})$,
$\tilde{T}_{2}(=\tilde{N}_{6}+\tilde{N}_{7}+\tilde{N}_{8}+\tilde{N}_{9}+2E_{2})$
and two double fibres $2F_{1}$, $2F_{2}$ of the fibration $f$.

\begin{enumerate}
\item
Suppose $|\tilde{\Gamma}_{2}|$ determines the elliptic fibration.
Consider an exact sequence
$0\longrightarrow\mathcal{O}_{W'}(\tilde{\Gamma}_{1}-E_{1})\longrightarrow\mathcal{O}_{W'}(\tilde{\Gamma}_{1})
\longrightarrow\mathcal{O}_{E_{1}}(\tilde{\Gamma}_{1})\longrightarrow
0.$ If we assume
$H^{0}(W',\mathcal{O}_{W'}(\tilde{\Gamma}_{1}-E_{1}))\neq0$, then
$\tilde{\Gamma}_{1}\equiv2E_{1}+\tilde{N}_{2}+\tilde{N}_{3}+
\tilde{N}_{4}+\tilde{N}_{5}+G\equiv\tilde{\Gamma}_{2}+G$ for some
effective divisor $G$, and so $p_{a}(G)=0$ because
$\tilde{\Gamma}_{2}G=2$. So there is an irreducible smooth curve $C$
with $p_{a}(C)=0$ (i.e. $C$ is a $(-2)$-curve) as a
component of $G$. We claim $C\tilde{N}_{i}=0$ for $i=2,3,\dots,9$.
Indeed, suppose $C\tilde{N}_{i}>0$ for some $i$, and then
$0=G\tilde{N}_{i}=(H+C)\tilde{N}_{i}$, where $G=H+C$ for some
effective divisor $H$. Since $H\tilde{N}_{i}<0$, $\tilde{N}_{i}$ is
a component of $H$. Thus
$\tilde{\Gamma}_{1}-\tilde{\Gamma}_{2}\equiv G=\tilde{N}_{i}+I$  for
some effective divisor $I$, which is impossible by
$p_{a}(\tilde{\Gamma}_{1})=2$, $p_{a}(\tilde{\Gamma}_{2})=1$,
$\tilde{\Gamma}_{2}I=2$, $\tilde{N}_{i}I=2$ and connectedness among
$\tilde{\Gamma}_{2},\tilde{N}_{i}$ and $I$ induced from Lemma
\ref{connectedness} since $\tilde{\Gamma}_{1}$ is nef and big. On
the other hand, suppose $C\tilde{N}_{i}<0$ for some $i$, then
$C=\tilde{N}_{i}$ because $C$ and $\tilde{N}_{i}$ are irreducible
and reduced. Thus $\tilde{\Gamma}_{1}-\tilde{\Gamma}_{2}\equiv
G=\tilde{N}_{i}+H$ for an effective divisor $H$, which is impossible
by $p_{a}(\tilde{\Gamma}_{1})=2$, $p_{a}(\tilde{\Gamma}_{2})=1$,
$\tilde{\Gamma}_{2}H=2$ and $\tilde{N}_{i}H=2$ and connectedness
among $\tilde{\Gamma}_{2},\tilde{N}_{i}$ and $H$ induced from Lemma
\ref{connectedness} since $\tilde{\Gamma}_{1}$ is nef and big. Hence
we have nine disjoint $(-2)$-curves
$C,\tilde{N}_{2},\dots,\tilde{N}_{9}$, which induce a contradiction
on the Enriques surface $W'$ by Lemma \ref{lemma:RSMN1}. Now, we
have $H^{0}(W',\mathcal{O}_{W'}(\tilde{\Gamma}_{1}-E_{1}))=0$, and
so
\[H^{0}(W',\mathcal{O}_{W'}(\tilde{\Gamma}_{1}))\longrightarrow
H^{0}(E_{1},\mathcal{O}_{E_{1}}(\tilde{\Gamma}_{1}))\] is an
injective map.

Since $h^{0}(W',\mathcal{O}_{W'}(\tilde{\Gamma}_{1}))=2$ and
$h^{0}(E_{1},\mathcal{O}_{E_{1}}(\tilde{\Gamma}_{1}))=2$ (because
$\tilde{\Gamma}_{1}E_{1}=1$),
$\tilde{\Gamma}_{1}\equiv\tilde{N}_{2}+\tilde{\Gamma}_{1}'$ for some
effective divisor $\tilde{\Gamma}_{1}'$; The injectivity of the
above map and $E_1=\mathbb P^1$ imply that the linear system
$\tilde{\Gamma}_{1}$ restricted on $E_1$ should move on $E_1$.
Therefore at least one member of the linear system of
$\tilde{\Gamma}_{1}$ should meet $\tilde N_2$.

Since $\tilde{\Gamma}_{1}$ is a smooth projective curve of genus 2
whose self intersection number is 2, and $\tilde{\Gamma}_{1}'\tilde
N_2=2$, we have $\tilde{\Gamma}_{1}'^2=0$ and
$p_a(\tilde{\Gamma}_{1}')=1$. And we note that
$h^{0}(W',\mathcal{O}_{W'}(\tilde{\Gamma}_{1}'))=1$. Therefore,
$|2\tilde{\Gamma}_{1}'|$ gives an elliptic fibration, and the
special member of $|2\tilde{\Gamma}_{1}'|$ contains $E_1$ because
$\tilde{\Gamma}_{1}'E_1=0$. Then this special member also contains
$\tilde N_3, \tilde N_4, \tilde N_5$ because
$\tilde{\Gamma}_{1}'\tilde N_i=0$ for $i=3, 4, 5$. Since
$|2\tilde{\Gamma}_{1}'|$ gives an elliptic fibration,
\[2\tilde{\Gamma}_{1}'\equiv 2E_1+\tilde N_3+\tilde N_4+ \tilde N_5
+ \tilde N\] where $\tilde N$ is a $(-2)$-curve with $\tilde NE_1=1,
\tilde N\tilde N_j=0$ for all $j=3, 4, 5, 6, 7, 8, 9$. And we get $\tilde N\tilde N_2
=2$ because $\tilde{\Gamma}_{1}'\tilde N_2=2$. Then we see that
$|2(\tilde N+\tilde N_2)|$ gives an elliptic pencil on $W'$.
On the other hand, by the classification of possible singular fibers on an elliptic
pencil on $W'$ (Theorem 5.6.2 in \cite{ESI}, or \cite{ES8}), we have that any elliptic fibration
on $W'$ has no singular fibers of type $2(\tilde N+\tilde N_2)$. We note that
$\tilde N_j$ for all $j=3, 4, 5, 6, 7, 8, 9$ are also on singular fibers.

\item
Suppose $|2\tilde{\Gamma}_{2}|$ determines the elliptic fibration.
Consider an exact sequence
\begin{displaymath}
\textrm{  }\textrm{  }\textrm{  }\textrm{  }
0\longrightarrow\mathcal{O}_{W'}(\tilde{\Gamma}_{1}-E_{1})\longrightarrow\mathcal{O}_{W'}(\tilde{\Gamma}_{1})
\longrightarrow\mathcal{O}_{E_{1}}(\tilde{\Gamma}_{1})\longrightarrow
0.
\end{displaymath}
 If we assume
$H^{0}(W',\mathcal{O}_{W'}(\tilde{\Gamma}_{1}-E_{1}))\neq0$, then
$\tilde{\Gamma}_{1}\equiv
E_{1}+\tilde{N}_{2}+\tilde{N}_{3}+\tilde{N}_{4}+\tilde{N}_{5}+G$ for
some effective divisor $G$ by the same reason as the above. Then it
is impossible by $p_{a}(\tilde{\Gamma}_{1})=2$, $E_{1}G=0$ and
$\tilde{N}_{i}G=1$ for all $i=2,3,4,5$ and connectedness among
$E_{1},\tilde{N}_{2},\tilde{N}_{3},\tilde{N}_{4},\tilde{N}_{5}$ and
$G$ induced from Lemma \ref{connectedness} since
$\tilde{\Gamma}_{1}$ is nef and big. Thus we have
$H^{0}(W',\mathcal{O}_{W'}(\tilde{\Gamma}_{1}-E_{1}))=0$, and so
\begin{displaymath}
H^{0}(W',\mathcal{O}_{W'}(\tilde{\Gamma}_{1}))\longrightarrow
H^{0}(E_{1},\mathcal{O}_{E_{1}}(\tilde{\Gamma}_{1}))
\end{displaymath}
is an injective map.

Since $h^{0}(W',\mathcal{O}_{W'}(\tilde{\Gamma}_{1}))=2$ and
$h^{0}(E_{1},\mathcal{O}_{E_{1}}(\tilde{\Gamma}_{1}))=3$ (because
$\tilde{\Gamma}_{1}E_{1}=2$),
$\tilde{\Gamma}_{1}\equiv\tilde{N}_{2}+\tilde{\Gamma}_{1}'$ for some
effective divisor $\tilde{\Gamma}_{1}'$ by the same reason as the
above. Then it is also impossible by the same argument as the above.
\end{enumerate}

\item[f)] The case $\substack{\Gamma_{0}\\(2,-2)}+\substack{\Gamma_{1}\\(2,0)}$\\
By Lemma \ref{criterion of alpha}, $\bar{E}_{1}\Gamma_{0}=2$ and $\bar{E}_{1}\Gamma_{1}=1$.
So we have $\tilde{\Gamma}_{0}\colon (4,6)$ and $\tilde{\Gamma}_{1}\colon(2,2)$ on the Enriques surface $W'$.

Consider an elliptic fibration of Enriques surface $f\colon
W'\longrightarrow \mathbb{P}^{1}$, and assume
$\tilde{\Gamma}_{1}F=2\gamma$, where $F$ is a general fibre of $f$.
Then $\gamma>0$ because $\tilde{\Gamma}_{1}$ cannot occur in a fibre
of $f$ since $p_a(\tilde\Gamma_1)=2$. Moreover, consider an exact
sequence
\[0\longrightarrow\mathcal{O}_{W'}(\tilde{\Gamma}_{1}-E_{1})\longrightarrow\mathcal{O}_{W'}(\tilde{\Gamma}_{1})
\longrightarrow\mathcal{O}_{E_{1}}(\tilde{\Gamma}_{1})\longrightarrow
0.\] If we assume
$H^{0}(W',\mathcal{O}_{W'}(\tilde{\Gamma}_{1}-E_{1}))\neq0$, then
$\tilde{\Gamma}_{1}\equiv
E_{1}+\tilde{N}_{2}+\tilde{N}_{3}+\tilde{N}_{4}+\tilde{N}_{5}+G$ for
some effective divisor $G$, which is impossible by $p_{a}(\tilde{\Gamma}_{1})=2$, $\tilde{N}_{i}G=1$
for all $i=2,3,4,5$ and connectedness among $E_{1},\tilde{N}_{2},\tilde{N}_{3},\tilde{N}_{4},\tilde{N}_{5}$
and $G$ induced from Lemma \ref{connectedness} since $\tilde{\Gamma}_{1}$ is nef and big. Now, we have
$H^{0}(W',\mathcal{O}_{W'}(\tilde{\Gamma}_{1}-E_{1}))=0$, and so
\begin{displaymath}
H^{0}(W',\mathcal{O}_{W'}(\tilde{\Gamma}_{1}))\longrightarrow
H^{0}(E_{1},\mathcal{O}_{E_{1}}(\tilde{\Gamma}_{1}))
\end{displaymath}
is an injective map. Since
$h^{0}(W',\mathcal{O}_{W'}(\tilde{\Gamma}_{1}))=2$ by Lemma
\ref{hgamma0=2} and
$h^{0}(E_{1},\mathcal{O}_{E_{1}}(\tilde{\Gamma}_{1}))=\gamma+1$
(because $\tilde{\Gamma}_{1}E_{1}=\gamma$),
$\tilde{\Gamma}_{1}\equiv\tilde{N}_{2}+\tilde{\Gamma}_{1}'$ for some
effective divisor $\tilde{\Gamma}_{1}'$ by the same reason as the
previous case. Then it is also impossible by the same argument as
the previous case.
\end{itemize}

Therefore, all other cases except
$B_0=\substack{\Gamma_{0}\\(3,0)}+\substack{\Gamma_{1}\\(1,-2)}$ or
$\substack{\Gamma_{0}\\(3,-2)}$ are
excluded.

\begin{lemma}\label{torsion}
If $W$ is birational to an Enriques surface then $S$ has a $2$-torsion element.
\end{lemma}
\begin{proof}
If $W$ is birational to an Enriques surface then $2K_W$ can be
written as $2A$ where $A$ is an effective divisor. Thus
$2K_V\equiv\tilde\pi^*(2A)+2\tilde R$, where $\tilde R$ is the
ramification divisor of $\tilde \pi$. So $G=\tilde\pi^*(A)+\tilde R$
is an effective divisor such that $G\sim K_V$ but $G\not\equiv K_V$
because $G$ is effective and $p_g(V)=0$. Since $2G\equiv 2K_V$,
$G-K_V$ is a 2-torsion element, and so $S$ has a 2-torsion element.
\end{proof}

\begin{rmk} \emph{Suppose
$B_0=\substack{\Gamma_{0}\\(3,0)}+\substack{\Gamma_{1}\\(1,-2)}$. By
Lemma \ref{criterion of alpha}, $\bar{E}_{1}\Gamma_{0}=2$ and
$\bar{E}_{1}\Gamma_{1}=1$. So we have $\tilde{\Gamma}_{0}\colon
(5,8)$, $\tilde{\Gamma}_{1}\colon (1,0)$ and
$\tilde{\Gamma}_{0}\tilde{\Gamma}_{1}=4$ on the Enriques surface
$W'$. We have $h^{0}(W',\mathcal{O}_{W'}(\tilde{\Gamma}_{0}))=5$
since $\tilde{\Gamma}_{0}\colon(5,8)$. However, the intersection
number $\tilde{\Gamma}_0\tilde{\Gamma}_1=4$ together with tangency
condition gives a six dimensional condition.}
\end{rmk}

By the results in Section 3 and 4, we have the table of
classification in Introduction.

\medskip

\section{Examples}\label{ex}
There is an example of a minimal surface $S$ of general type with
$p_{g}(S)=q(S)=0$, $K_{S}^{2}=7$ with an involution. Such an example
can be found in Example $4.1$ of \cite{BS}. Since the surface $S$ is
constructed by bidouble cover (i.e. $\mathbb{Z}_{2}^{2}$-cover),
there are three involutions $\gamma_{1},\gamma_{2}$ and $\gamma_{3}$
on $S$. The bicanonical map $\varphi$ is composed with the
involution $\gamma_{1}$ but not with $\gamma_{2}$ and $\gamma_{3}$.
Thus the pair $(S,\gamma_{1})$ has $k=11$ by
Proposition \ref{coro:3.6}, and then $W_{1}$ is rational and $K_{W_{1}}^{2}=-4$ by Theorem
\ref{theorem:k} $(ii)$, where $W_{1}$ is the blowing-up of all the nodes in
$S/\gamma_{1}$. On the other hand, we cannot see directly about $k$,
$K^{2}$ and the Kodaira dimension of the quotients in the case
$(S,\gamma_{2})$ and $(S,\gamma_{3})$. We use the notation of Example
$4.1$ in \cite{BS}, but $P$ denotes $\Sigma$. Moreover, $W_{i}$ comes from the blowing-up at all the
nodes of $\Sigma_{i}:=S/\gamma_{i}$ for $i=1,2,3$.

We now observe that $W_{i}$ is constructed by using a double
covering $T_{i}$ of a rational surface $P$ with a branch divisor
related to $L_{i}$. The surface $P$ is obtained as the blowing-up at
six points on a configuration of lines in $\mathbb{P}^{2}$. The
surface $W_{i}$ is obtained by examining $(-1)$ and $(-2)$-curves on
$T_{i}$ and contracting some of them.

We will now explain this examination in more details for each case.
Firstly, for $i=1$, then $K_{T_{1}}^{2}=-6$ since
$K_{T_{1}}\equiv\pi_{1}^{*}(K_{P}+L_{1})$, where $\pi_{1}\colon
T_{1}\longrightarrow P$ is the double cover. We observe that there
are only two $(-1)$-curves on $T_{1}$ because $S_{3},S_{4}$ are on
the branch locus of $\pi_{1}$. So
$K_{W_{1}}^{2}=K_{\Sigma_{1}}^{2}=-6+2=-4$. On the other hand, we
also observe that there are only seven nodes and four $(-2)$-curves
on $T_{1}$ because $D_{2}D_{3}=7$ and $S_{1}$ and $S_{2}$ do not
contain in $D_{2}+D_{3}$. So $\Sigma_{1}$ has $k=11$ nodes.
Moreover,
$H^{0}(T_{1},\mathcal{O}_{T_{1}}(2K_{T_{1}}))=H^{0}(P,\mathcal{O}_{P}(2K_{P}+2L_{1}))
\oplus H^{0}(P,\mathcal{O}_{P}(2K_{P}+L_{1}))$ since
$2K_{T_{1}}\equiv\pi_{1}^{*}(2K_{P}+2L_{1})$ and
${\pi_{1}}_{*}(\mathcal{O}_{T_{1}})=\mathcal{O}_{P}\oplus
\mathcal{O}_{P}(-L_{1})$. So
$H^{0}(T_{1},\mathcal{O}_{T_{1}}(2K_{T_{1}}))=0$ because
$2K_{P}+2L_{1}=4l-2e_{2}-4e_{4}-2e_{5}-2e_{6}$
 and $2K_{P}+L_{1}=-l+e_{1}+e_{3}-e_{4}$. This means that $T_{1}$ is rational,
 and therefore $W_{1}$ is rational. For the branch divisor $B_{0}$,
 we observe $f_{2}$ and $\Delta_{1}$ in $D_{1}$. Since $f_{2}D_{2}=4$
 and $f_{2}D_{3}=4$, $f_{2}(D_{2}+D_{3})=8$. By Hurwitz's formula,
 $2p_{g}(\Gamma_{0})-2=2(2p_{g}(f_{2})-2)+8$, and so $p_{g}(\Gamma_{0})=3$
 because $f_{2}$ is rational, and moreover $\Gamma_{0}^{2}=0$ because $f_{2}^{2}=0$.
 This means $\Gamma_{0}\colon(3,0)$. Similarly, since $\Delta_{1}D_{2}=1$
 and $\Delta_{1}D_{3}=5$, $\Delta_{1}(D_{2}+D_{3})=6$. By Hurwitz's formula,
 $2p_{g}(\Gamma_{1})-2=2(2p_{g}(\Delta_{1})-2)+6$, and so $p_{g}(\Gamma_{1})=2$
 because $\Delta_{1}$ is rational, and moreover $\Gamma_{1}^{2}=-2$ because
 $\Delta_{1}^{2}=-1$. This means $\Gamma_{1}\colon(2,-2)$,
 thus $B_{0}=\substack{\Gamma_{0}\\(3,0)}+\substack{\Gamma_{1}\\(2,-2)}$.

Secondly, in the case $i=2$, we calculate $K_{T_{2}}^{2}=-6$. We
observe that there are only four $(-1)$-curves on $T_{2}$ because
$S_{1},S_{2},S_{3},S_{4}$ are on the branch locus. So
$K_{W_{2}}^{2}=K_{\Sigma_{2}}^{2}=-6+4=-2$. On the other hand, we
also observe that there are only nine nodes on $T_{2}$ because
$D_{1}D_{3}=9$. So $\Sigma_{2}$ has $k=9$ nodes. Yifan Chen \cite{Chen} shows that
$H^{0}(T_{2},\mathcal{O}_{T_{2}}(2K_{T_{2}}))=1$  and that $W_{2}$ is birational to an Enriques surface.
For the branch
divisor $B_{0}$, we observe $f_{3}$ and $\Delta_{2}$ in $D_{2}$.
Since $f_{3}D_{1}=2$ and $f_{3}D_{3}=6$, $p_{g}(\Gamma_{0})=3$
because $f_{3}$ is rational, and $\Gamma_{0}^{2}=0$ because
$f_{3}^{2}=0$. This means $\Gamma_{0}\colon(3,0)$. Moreover, since
$\Delta_{2}D_{1}=3$ and $\Delta_{2}D_{3}=1$, $p_{g}(\Gamma_{1})=1$
because $\Delta_{2}$ is rational, and $\Gamma_{1}^{2}=-2$ because
$\Delta_{2}^{2}=-1$. This means $\Gamma_{1}\colon(1,-2)$, thus
$B_{0}=\substack{\Gamma_{0}\\(3,0)}+\substack{\Gamma_{1}\\(1,-2)}$.

Lastly, for $i=3$, we get $K_{T_{3}}^{2}=-4$. There are only two
$(-1)$-curves on $T_{3}$ because $S_{1},S_{2}$ are on the branch
locus. So $K_{W_{3}}^{2}=K_{\Sigma_{3}}^{2}=-4+2=-2$. On the other
hand, there are only nine nodes on $T_{3}$ because $D_{1}D_{2}=5$
and $S_{3}$ and $S_{4}$ do not contain in $D_{1}+D_{2}$. So
$\Sigma_{2}$ has $k=9$ nodes. Also,
$H^{0}(T_{3},\mathcal{O}_{T_{3}}(2K_{T_{3}}))=0$ by a similar
argument to the case $i=1$. So $W_{3}$ is rational. For the branch
divisor $B_{0}$, we observe $f_{1},f'_{1}$ and $\Delta_{3}$ in
$D_{3}$. Since $f_{1}D_{1}=4$ and $f_{1}D_{2}=2$,
$p_{g}(\Gamma_{0})=2$ because $f_{1}$ is rational, and
$\Gamma_{0}^{2}=0$ because $f_{1}^{2}=0$. This means
$\Gamma_{0}\colon(2,0)$, and $\Gamma_{1}$ related to $f'_{1}$ is
also of type $(2,0)$. Moreover, since $\Delta_{3}D_{1}=1$ and
$\Delta_{3}D_{2}=3$, $p_{g}(\Gamma_{2})=1$ because $\Delta_{3}$ is
rational, and $\Gamma_{2}^{2}=-2$ because $\Delta_{3}^{2}=-1$. This
means $\Gamma_{2}\colon(1,-2)$, thus
$B_{0}=\substack{\Gamma_{0}\\(2,0)}+\substack{\Gamma_{1}\\(2,0)}+\substack{\Gamma_{2}\\(1,-2)}$.

\medskip

The following table summaries the above computation:
\begin{table}[h]
\centering
\begin{tabular}{|c|c|c|c|c|}
\hline
                   & $k$      &  $K_{W_{i}}^{2}$ & $B_{0}$ & $W_{i}$\\
\hline
$(S,\gamma_{1})$  &  $11$     &  $-4$            & $\substack{\Gamma_{0}\\(3,0)}+\substack{\Gamma_{1}\\(2,-2)}$ & rational\\
\hline
$(S,\gamma_{2})$  &  $9$      &  $-2$            & $\substack{\Gamma_{0}\\(3,0)}+\substack{\Gamma_{1}\\(1,-2)}$ & birational to an Enriques surface \\
\hline
$(S,\gamma_{3})$  &  $9$      &  $-2$            &
$\substack{\Gamma_{0}\\(2,0)}+\substack{\Gamma_{1}\\(2,0)}+\substack{\Gamma_{2}\\(1,-2)}$ & rational\\
\hline
\end{tabular}
\end{table}

\begin{rmkN}\emph{
In the pre-version of the paper, the 3 quotients of Inoue's example were claimed rational surfaces. 
And we raised the question for the existence of a minimal smooth projective surface of general type 
with $p_g=0$ and $K^2=7$ which is a double cover of a surface birational to an Enriques surface or 
a surface of general type. Carlos Rito \cite{Rito} constructed an example with Enriques quotient. 
Later Yifan Chen \cite{Chen} showed that Rito's example is the Inoue's one, and Carlos Rito verified 
that one of quotients is not rational but is an Enriques.}
\end{rmkN}

\medskip

{\em Acknowledgements}. Both authors would like to thank Margarida
Mendes Lopes for sharing her ideas which run through this work. The
simplified proof of Theorem \ref{theorem:k} and Lemma \ref{torsion}
are due to her. And they would like to thank Yifan Chen, Stephen Coughlan,
JongHae Keum, Miles Reid, and Carlos Rito for some useful comments. The authors
thank a referee for valuable comments to modify the original version.

The first named author was partially supported by the Special
Research Grant of Sogang University. Both authors were partially
supported by the World Class University program through the National
Research Foundation of Korea funded by the Ministry of Education,
Science and Technology (R33-2008-000-10101-0).

\bigskip

\begin{small}\end{small}

\end{document}